\documentclass[11pt,a4paper]{article}

\usepackage{graphics}
\usepackage{amsthm}
\usepackage{enumerate}
\usepackage{fullpage}
\usepackage[colorlinks,linkcolor=blue,citecolor=blue]{hyperref}

\usepackage{amsmath}
\usepackage{amsfonts}
\usepackage[linesnumbered,lined,boxed,ruled,vlined,algo2e]{algorithm2e}
\usepackage{subfig}
\usepackage[heightadjust=all,valign=c]{floatrow}

\newtheorem{thm}{Theorem}
\newtheorem{lem}[thm]{Lemma}
\newtheorem{cor}[thm]{Corollary}

\theoremstyle{definition}
\newtheorem{rmk}[thm]{Remark}

\newcommand\BII[1]{\mathit{II}_{#1}}
\newcommand\BIII[1]{\mathit{III}_{#1}}
\def\eqref#1{(\ref{#1})}

\title{Essential obstacles to Helly circular-arc graphs}

\author{Mart\'{\i}n D.~Safe\thanks{Departamento de Matem\'atica, Universidad Nacional del Sur (UNS), Bah\'ia Blanca, Argentina and INMABB, Universidad Nacional del Sur (UNS)-CONICET, Bah\'ia Blanca, Argentina. E-mail address: \texttt{msafe@uns.edu.ar}}}

\begin{document}

\maketitle

\abstract{A Helly circular-arc graph is the intersection graph of a set of arcs on a circle having the Helly property. We introduce essential obstacles, which are a refinement of the notion of obstacles, and prove that essential obstacles are precisely the minimal forbidden induced circular-arc subgraphs for the class of Helly circular-arc graphs. We show that it is possible to find in linear time, in any given obstacle, some minimal forbidden induced subgraph for the class of Helly circular-arc graphs contained as an induced subgraph. Moreover, relying on an existing linear-time algorithm for finding induced obstacles in circular-arc graphs, we conclude that it is possible to find in linear time an induced essential obstacle in any circular-arc graph that is not a Helly circular-arc graph. The problem of finding a forbidden induced subgraph characterization, not restricted only to circular-arc graphs, for the class of Helly circular-arc graphs remains unresolved. As a partial answer to this problem, we find the minimal forbidden induced subgraph characterization for the class of Helly circular-arc graphs restricted to graphs containing no induced claw and no induced $5$-wheel. Furthermore, we show that there is a linear-time algorithm for finding, in any given graph that is not a Helly circular-arc graph, an induced subgraph isomorphic to claw, $5$-wheel, or some minimal forbidden induced subgraph for the class of Helly circular-arc graphs.}

\section{Introduction}

The \emph{intersection graph} of a set $\mathcal A$ of arcs on a circle is a graph having one vertex for each arc in $\mathcal A$ and such that two different vertices are adjacent if and only if the corresponding arcs have nonempty intersection. A graph $G$ is a \emph{circular-arc graph}~\cite{MR0309810} if $G$ is the intersection graph of some set $\mathcal A$ of arcs on a circle; if so, the set $\mathcal A$ is called a \emph{circular-arc model} of $G$. Forbidden structures for the class of circular-arc graphs and its main subclasses, as well as efficient algorithms for finding such structures, have received a great deal of attention~\cite{MR1271882,MR2536884,MR3575014,MR1773654,MR3191588,MR3451138,MR2071482,MR2765574,MR2554791,MR2428582,MR3030589,MR2399252,MR2538982,MR4043758,MR3318808,MR0450140,MR0379298}. A complete characterization by forbidden structures for the class of circular-arc graphs, together with an $O(n^3)$-time algorithm for finding one such forbidden structure in any given graph that is not a circular-arc graph, was given in~\cite{MR3451138}. Two surveys on structural results regarding circular-arc graphs appeared in~\cite{MR3159129,MR2567965}. Linear-time recognition algorithms for circular-arc graphs were proposed in~\cite{MR2283288,MR1993248}.

A family of sets has the \emph{Helly property}~\cite{MR0357172}, or simply is \emph{Helly}, if every nonempty subfamily of pairwise intersecting sets has nonempty total intersection. A \emph{Helly circular-arc graph} (sometimes also \emph{$\Theta$ circular-arc graph}) is a circular-arc graph admitting a circular-arc model that has the Helly property. These graphs were introduced by Gavril in~{\cite{MR0376439}, where he derived an $O(n^3)$-time recognition algorithm for them based on the circular-ones property for columns of their clique-matrices. Based on the same property, a linear-time algorithm for testing isomorphism of Helly circular-arc graphs was devised in~\cite{MR3040544}; a parallel algorithm~\cite{DBLP:journals/tpds/Chen96} and a logspace algorithm~\cite{MR3463413} for the same task are also known. Clique graphs of Helly circular-arc graphs were studied in \cite{MR2215017,MR1849458,MR2570626,MR2652002}. In~\cite{MR2765574}, Joeris, Lin, McConnell, Spinrad, and Szwarcfiter gave a linear-time recognition algorithm for Helly circular-arc graphs.\footnote{Some results of \cite{MR2765574} appeared also in the extended abstract~\cite{MR2304709}.} Moreover, if the input is a circular-arc graph, their algorithm is \emph{certifying}~\cite{MR2247731,mcconnell2011certifying}, meaning that it produces an easy-to-check certificate for the correctness of its answer. Namely, if the input is a Helly circular-arc graph, their algorithm answers `yes' together with a \emph{positive certificate}, which consists of a Helly circular-arc model of the input graph; otherwise, the answer is `no' together with a \emph{negative certificate}, which consists of an induced subgraph of the input graph that belongs to a family of graphs called \emph{obstacles}~\cite{MR2765574}. That an induced obstacle serves as a certificate of the `no' answer follows from the structural result below. The precise definition of obstacles is given in Section~\ref{sec:main}, while more basic definitions are given in Section~\ref{sec:defs}.

\begin{thm}[\cite{MR2765574}]\label{thm:HCAobst} A circular-arc graph $G$ is a Helly circular-arc graph if and only if $G$ contains no induced obstacle.\end{thm}

The above theorem gives a characterization of Helly circular-arc graphs by forbidden induced subgraphs restricted to circular-arc graphs. However, this characterization is not by minimal forbidden induced subgraphs. In fact, there are obstacles that contain other obstacles as induced subgraphs (e.g., $\overline{2P_4}$ and $\overline{2C_5}$ are obstacles such that the former is an induced subgraph of the latter~\cite{MR3232456}). Moreover, some obstacles are not circular-arc graphs (e.g., $\overline{C_6}$~\cite{MR3232456} and $\overline{C_5+K_2}$, where $+$ denotes disjoint union) and thus cannot occur as induced subgraphs of any circular-arc graph.

We say an obstacle is \emph{minimal} if it contains no induced obstacle having fewer vertices. A \emph{minimal circular-arc obstacle} is an obstacle that is both minimal and a circular-arc graph. Clearly, replacing `obstacle' by `minimal circular-arc obstacle' in Theorem~\ref{thm:HCAobst}, yields the characterization for the class of Helly circular-arc graphs by minimal forbidden induced subgraphs restricted to circular-arc graphs. A partial list of minimal circular-arc obstacles was given in~\cite{MR3232456}. In this work, we introduce essential obstacles, a refinement of the notion of obstacles, and prove that essential obstacles are precisely the minimal circular-arc obstacles or, equivalently, the minimal forbidden induced circular-arc subgraphs for the class of Helly circular-arc graphs, where by a  \emph{circular-arc subgraph} we mean a subgraph which is a circular-arc graph.

\begin{thm}\label{thm:main} The minimal forbidden induced circular-arc subgraphs for the class of Helly circular-arc graphs are precisely the essential obstacles.\end{thm}

Moreover, we show that, given any obstacle, it is possible to find in linear time a minimal forbidden induced subgraph for the class of Helly circular-arc graphs contained in it as an induced subgraph. Hence, given any negative certificate produced by Joeris et al.'s algorithm, it is possible to obtain a minimal negative certificate while preserving the linear time bound.

The problem of finding a forbidden induced subgraph characterization, not restricted only to circular-arc graphs, for the class of Helly circular-arc graphs remains unresolved; i.e., no analog of Theorem~\ref{thm:HCAobst} where `A circular-arc graph $G$' is replaced by just `A graph $G$' is known. As a partial answer to this problem, we obtain the minimal forbidden induced subgraph characterization for the class of Helly circular-arc graphs restricted to graphs containing no induced claw and no induced $5$-wheel, where the \emph{claw} is the complete bipartite graph $K_{1,3}$ and the \emph{$5$-wheel} is the graph that arises from a chordless cycle on $5$ vertices by adding one vertex adjacent to all vertices of the cycle. Moreover, we show that it is possible to find in linear time an induced claw, an induced $5$-wheel, or an induced minimal forbidden induced subgraph for the class of Helly circular-arc graphs in any given graph that is not a Helly circular-arc graph. Notice that although the minimal forbidden induced subgraph characterization for circular-arc graphs is known restricted to complements of bipartite graphs~\cite{MR0450140} and to claw-free chordal graphs~\cite{MR2536884}, no forbidden induced subgraph characterization for circular-arc graphs restricted to the larger class of graphs containing no induced claw and no induced $5$-wheel is known.

This work is organized as follows. In Section~\ref{sec:defs}, we give some definitions and preliminaries. In Section~\ref{sec:main}, we introduce essential obstacles, we prove that essential obstacles are precisely the minimal circular-arc obstacles, we show that it is possible to find in linear time, in any given obstacle, some minimal forbidden induced subgraph for the class of Helly circular-arc graphs contained as an induced subgraph, and we conclude that it is possible to find in linear time an induced essential obstacle in any given circular-arc graph that is not a Helly circular-arc graph. In Section~\ref{sec:quasi-line}, we give the minimal forbidden induced subgraph characterization of Helly circular-arc graphs restricted to graphs containing no induced claw and no induced $5$-wheel and show that it is possible to find in linear time, in any given graph that is not a Helly circular-arc graph, an induced subgraph isomorphic to claw, $5$-wheel, or some minimal forbidden induced subgraph for the class of Helly circular-arc graphs.

\section{Definitions and preliminaries}\label{sec:defs}

All graphs in this work are finite, undirected, and with no loops or multiple edges. For each positive integer $k$, we denote by $[k]$ the set $\{1,2,\ldots,k\}$. For any graph-theoretic notions not defined here, the reader is referred to~\cite{MR1367739}.

Let $G$ be a graph. We denote by $V(G)$ and $E(G)$ the vertex set and the edge set of $G$, respectively. If $v$ is vertex of $G$, the \emph{neighborhood} of $v$ in $G$, denoted $N_G(v)$, is the set of vertices of $G$ adjacent to $v$, whereas the \emph{closed neighborhood} of $v$ in $G$ is the set $N_G[v]=N_G(v)\cup\{v\}$. We denote by $\overline N_G(v)$ the set of vertices of $G$ different from $v$ and nonadjacent to $v$. The \emph{complement of $G$}, denoted $\overline G$, is the graph with the same vertex set as $G$ and such that two of its vertices are adjacent in $\overline G$ if and only if they are nonadjacent in $G$. Thus, $\overline N_G(v)=N_{\overline G}(v)$ for every vertex $v$ of $G$. If $X\subseteq V(G)$, the \emph{subgraph of $G$ induced by $X$}, denoted $G[X]$, is the graph having $X$ as vertex set and whose edges are those edges of $G$ having both endpoints in $X$. If $X\neq V(G)$, $G[X]$ is called a \emph{proper induced subgraph} of $G$. We say that \emph{$G$ contains an induced} (resp.\ \emph{contains a proper induced}) $H$ if $H$ is isomorphic to some induced subgraph (resp. proper induced subgraph) of $G$. If $W\subseteq V(G)$, we denote by $G-W$ the graph $G[V(G)-W]$. If $v\in V(G)$, we denote $G-\{v\}$ simply by $G-v$. If $\mathcal H$ is a set of graphs, we say that $G$ is {$\mathcal H$-free} if $G$ contains no induced $H$ for any graph $H$ in the set $\mathcal H$. If $H$ is a graph, we write \emph{$H$-free} to mean $\{H\}$-free. Let $\mathcal G$ be a \emph{hereditary graph class} (i.e., $\mathcal G$ is closed under taking induced subgraphs). A \emph{minimal forbidden induced subgraph for the class $\mathcal G$} is any graph $G$ that does not belong to $\mathcal G$ but such that every proper induced subgraph of $G$ belongs to $\mathcal G$. A \emph{clique} of $G$ is a set of pairwise adjacent vertices of $G$. We say a clique is \emph{maximal} to mean that it is inclusion-wise maximal. Two subsets $U$ and $W$ of $V(G)$ are \emph{complete} (resp.\ \emph{anticomplete}) if $U$ and $W$ are disjoint and each vertex of $U$ is adjacent (resp.\ nonadjacent) to each vertex of $W$. We say a vertex $v$ of $G$ is \emph{complete} (resp.\ \emph{anticomplete}) to a subset $W$ of $V(G)$ if $\{v\}$ is complete (resp.\ anticomplete) to $W$.

A \emph{chord} of a path (resp.\ a cycle) is an edge joining two nonconsecutive vertices of the path (resp.\ the cycle). A path or cycle is \emph{chordless} if it has no chord. If $n$ is a positive integer, we denote by $P_n$, $C_n$, and $K_n$, the chordless path, the chordless cycle, and the complete graph on $n$ vertices, respectively. We denote the complete bipartite graph with partite sets of sizes $k_1$ and $k_2$ by $K_{k_1,k_2}$. The \emph{$k$-wheel} is the graph that arises from $C_k$ by adding one vertex adjacent to all its vertices.

If $G$ and $H$ are two vertex-disjoint graphs, the \emph{disjoint union of $G$ and $H$}, denoted $G+H$, is the graph with vertex set $V(G)\cup V(H)$ and edge set $E(G)\cup E(H)$. If $G$ is a graph and $k$ is a nonnegative integer, we denote by $kG$ the disjoint union of $k$ graphs, each of which isomorphic to $G$. If $G$ is a graph, we denote by $G^*$ the graph $G+K_1$. The graphs $C_4^*$, $K_{2,3}$, domino, $G_3$, $\overline{C_6}$, and $\overline{C_5+K_2}$, which are depicted in Figure~\ref{fig:forb-HCA}, are some minimal forbidden induced subgraphs for the class of Helly circular-arc graphs. None of these six graphs is a circular-arc graph.

Let $\mathcal A$ be a circular-arc model on a circle $C$ of a graph $G$. If $Q$ is a clique of $G$, a \emph{clique point of $Q$ in $\mathcal A$} is any point of $C$ that belongs to all those arcs in $\mathcal A$ corresponding to vertices of $Q$. Hence, $\mathcal A$ is Helly if and only if each maximal clique of $G$ has a clique point in $\mathcal A$.

\begin{figure}[t!]
\ffigbox[\textwidth]{%
\ffigbox[\FBwidth]{%
\begin{subfloatrow}
\ffigbox[0.13\textwidth]{\includegraphics{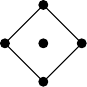}}{\caption{$C_4^*$}}
\ffigbox[0.13\textwidth]{\includegraphics{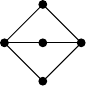}}{\caption{$K_{2,3}$}}
\ffigbox[0.13\textwidth]{\includegraphics{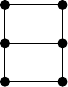}}{\caption{domino}}
\ffigbox[0.13\textwidth]{\includegraphics{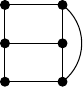}}{\caption{$G_3$}}
\ffigbox[0.13\textwidth]{\includegraphics{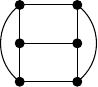}}{\caption{$\overline{C_6}$}}
\ffigbox[0.13\textwidth]{\includegraphics{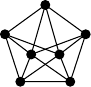}}{\caption{$\overline{C_5+K_2}$}}
\end{subfloatrow}}{}
}{\caption{Some minimal forbidden subgraphs for the class of Helly circular-arc graphs}\label{fig:forb-HCA}}
\end{figure}

If $\mathcal L$ is a circular or linear enumeration of vertices, we denote the set of vertices occurring in $\mathcal L$ by $V(\mathcal L)$. When discussing algorithms, we use $n$ and $m$ to denote the number of vertices and edges of the input graph, respectively. An algorithm taking a graph as input is \emph{linear-time} if it can be carried out in at most $O(n+m)$ time. We will also consider algorithms whose input is a circular-arc model. In such cases, we denote by $n$ the number of arcs in the model and we assume that the $2n$ extremes of these arcs are pairwise different and are given in the order in which they occur in some traversal of the circle.

We define a \emph{pseudo-domino} as any graph $D$ with vertex set $\{a_1,a_2,b_1,b_2,c_1,c_2\}$ such that $a_1a_2$, $b_1b_2$, $c_1c_2$, $a_1b_1$, $a_2b_2$, $b_1c_1$, and $b_2c_2$ are edges of $D$ and $a_1b_2$, $a_2b_1$, $b_1c_2$, and $b_2c_1$ are nonedges of $D$. We call the unordered pairs $a_1c_1$ and $a_2c_2$ the \emph{handles} of $D$ and the unordered pairs $a_1c_2$ and $a_2c_1$ the \emph{diagonals} of $D$. Notice that handles and diagonals may or may not be edges of $D$. The following lemma will be useful in the next section.

\begin{lem}\label{lem:pseudo-domino} If $D$ is a pseudo-domino, then one of the following assertions holds:
\begin{enumerate}[(i)]
 \item $D$ contains an induced $K_{2,3}$;
 \item $D$ is isomorphic to domino, $G_3$, or $\overline{C_6}$;
 \item both handles and at least one of the diagonals of $D$ are edges of $D$.
\end{enumerate}\end{lem}
\begin{proof} We say a graph $P$ is a \emph{pseudo-flag} if $P$ has vertex set $\{a_1,a_2,b_1,b_2,c\}$, $a_1a_2$, $b_1b_2$, $a_1b_1$, $a_2b_2$, and $b_1c$ are edges of $P$ and $a_1b_2$, $a_2b_1$, and $b_2c$ are nonedges of $P$. The \emph{handle} of $P$ is the unordered pair $a_1c$ and the \emph{diagonal} of $P$ is the unordered pair $a_2c$. Clearly, either the handle of $P$ is an edge of $P$ whenever the diagonal of $P$ is an edge of $P$, or $P$ is isomorphic to $K_{2,3}$. Let $D$ be a pseudo-domino. As $D-a_1$, $D-a_2$, $D-c_1$, and $D-c_2$ are induced pseudo-flags of $D$, if at least one diagonal of $D$ is an edge of $D$, then either $D$ contains an induced $K_{2,3}$ or both handles of $D$ are edges of $D$. If, on the contrary, no diagonal of $D$ is an edge of $D$, then $D$ is isomorphic to domino, $G_3$, or $\overline{C_6}$.\end{proof}

\section{Essential obstacles}\label{sec:main}

An \emph{obstacle enumeration in a graph $G$} is a circular enumeration $\mathcal Q=v_1,v_2,\ldots,v_k$ of $k\geq 3$ pairwise different vertices such that $Q=\{v_1,\ldots,v_k\}$ is a clique of $G$ and, for each $i\in[k]$, a linear enumeration $\mathcal W_i$ consisting of one or two vertices of $G$ such that one of the following conditions holds:
\begin{enumerate}
 \item[$(\mathcal O_1)$] $\mathcal W_i=w_i$ where $\overline N_G(w_i)\cap Q=\{v_i,v_{i+1}\}$;
 \item[$(\mathcal O_2)$] $\mathcal W_i=u_i,z_i$ where $\overline N_G(u_i)\cap Q=\{v_i\}$, $\overline N_G(z_i)\cap Q=\{v_{i+1}\}$, and $u_iz_i\in E(G)$;
\end{enumerate}
where here, and henceforth, all subindices on $u_i$, $v_i$, $z_i$, $w_i$, and $\mathcal W_i$ are modulo $k$. (Recall the notation $\overline N_G(v)=N_{\overline G}(v)$ introduced in the preceding section.) The clique $Q$ is called the \emph{core of $\mathcal Q$}. For each $i\in[k]$, $\mathcal W_i$ is called the \emph{witness enumeration of $v_iv_{i+1}$ in $\mathcal Q$}. The linear enumerations $\mathcal W_1,\ldots,\mathcal W_k$ are called the \emph{witness enumerations of $\mathcal Q$} and the vertices in the set $W(\mathcal Q)=V(\mathcal W_1)\cup\cdots\cup V(\mathcal W_k)$ are called the \emph{witnesses of $\mathcal Q$}. Whenever we refer to an obstacle enumeration $\mathcal Q=v_1,v_2\ldots,v_k$, some specific witness enumeration for each of the edges $v_1v_2$, $v_2v_3$, $\ldots$, $v_kv_1$ is implicit. An \emph{obstacle}~\cite{MR2765574} is a graph $G$ such that $V(G)=V(\mathcal Q)\cup W(\mathcal Q)$ for some obstacle enumeration $\mathcal Q$ in $G$; if so, we will say that \emph{$\mathcal Q$ is an obstacle enumeration of $G$}.

A partial list of minimal circular-arc obstacles was given in~\cite{MR3232456}. For each $k\geq 3$, the \emph{complete $k$-sun}, denoted $S_k$, is the graph on $2k$ vertices $w_1,\ldots,w_k,v_1,\ldots,v_k$, where $\{w_1,\ldots,w_k\}$ is a clique and, for each $i\in[k]$, the only neighbors of $v_i$ are $w_{i-1}$ and $w_i$. The graphs $1$-pyramid, $2$-pyramid, and $U_4$ are depicted in Figure~\ref{fig:pyrs+U4}.

\begin{figure}[t!]
\ffigbox[\textwidth]{%
\ffigbox[\FBwidth]{%
\begin{subfloatrow}
\ffigbox[0.18\textwidth]{\includegraphics{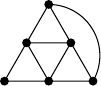}}{\caption{$1$-pyramid}}
\ffigbox[0.18\textwidth]{\includegraphics{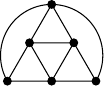}}{\caption{$2$-pyramid}}
\ffigbox[0.18\textwidth]{\includegraphics{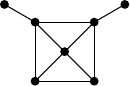}}{\caption{$U_4$}}
\end{subfloatrow}}{}
}{\caption{The graphs $1$-pyramid and $2$-pyramid and the minimal circular-arc obstacle $U_4$}\label{fig:pyrs+U4}}
\end{figure}

\begin{thm}[\cite{MR3232456}] The minimal circular-arc obstacles containing no induced $1$-pyramid and no induced $2$-pyramid are $\overline{3K_2}$, $U_4$, and $\overline{S_k}$ for each $k\geq 3$.\end{thm}

In this section, we will give a precise description of all minimal circular-arc obstacles. For that purpose, we need some specific definitions. Let $\mathcal Q$ be an obstacle enumeration in a graph $G$ and let $Q=V(\mathcal Q)$. For each witness $y$ of $\mathcal Q$, let $\ell_{\mathcal Q}(y)$ and $r_{\mathcal Q}(y)$ denote the vertices of $\mathcal Q$ such that $\overline N_G(y)\cap Q=\{\ell_{\mathcal Q}(y),r_{\mathcal Q}(y)\}$ and either $\ell_{\mathcal Q}(y)=r_{\mathcal Q}(y)$ or $r_{\mathcal Q}(y)$ occurs immediately after $\ell_{\mathcal Q}(y)$ in $\mathcal Q$. We will usually denote $\ell_{\mathcal Q}(y)$ and $r_{\mathcal Q}(y)$ simply by $\ell(y)$ and $r(y)$, respectively. Given an edge $y_1y_2$ of $G$ joining two witnesses $y_1$ and $y_2$ of $\mathcal Q$:
\begin{itemize}
 \item We say the ordered pair $(y_1,y_2)$ is an \emph{inner-shortcut pair} if $\ell(y_1)\notin\overline N_G(y_2)$, $r(y_2)\notin\overline N_G(y_1)$, and $\ell(y_1)$ and $r(y_2)$ are nonconsecutive in $\mathcal Q$. We say the edge $y_1y_2$ is an \emph{inner shortcut of $\mathcal Q$} if at least one of $(y_1,y_2)$ and $(y_2,y_1)$ is an inner-shortcut pair.
 
 \item We say the edge $y_1y_2$ is an \emph{outer shortcut} if there are two vertices $q_1$ and $q_2$ that occur consecutively in $\mathcal Q$ such that $\overline N_G(y_1)\cap Q=\{q_1\}$ and $\overline N_G(y_2)\cap Q=\{q_2\}$ but $y_1$ and $y_2$ do not occur together in any witness enumeration of $\mathcal Q$.
 
 \item We say the edge $y_1y_2$ is a \emph{shortcut} if it is an inner shortcut or an outer shortcut.
 
 \item We say the edge $y_1y_2$ is a \emph{cover} if $\overline N_G(y_1)\cup\overline N_G(y_2)\supseteq Q$.
 
 \item We say the edge $y_1y_2$ is \emph{valid} if either $\overline N_G(y_1)\cap Q$ and $\overline N_G(y_2)\cap Q$  are comparable (i.e., one is a subset of the other) or $y_1$ and $y_2$ occur together in some witness enumeration of $\mathcal Q$.
 
 Roughly speaking, if the witnesses of $\mathcal Q$ are labeled as in the definition of obstacles, then the edge $y_1y_2$ is valid if and only if $y_1y_2$ equals $z_{i-1}w_i$, $w_{i-1}u_i$, $z_{i-1}u_i$, or $u_iz_i$ for some $i\in[k]$.
\end{itemize}
Clearly, at most one of the following assertions holds: (i) $y_1y_2$ is an inner shortcut; (ii) $y_1y_2$ is an outer shortcut; (iii) $y_1y_2$ is a cover; (iv) $y_1y_2$ is valid. We will prove later on (Lemma~\ref{lem:noshort+nocover=essential}) that, actually, one of these assertions does hold.

We say an obstacle enumeration $\mathcal Q$ is \emph{essential} if every edge joining two of its witnesses is valid. An obstacle $G$ is \emph{essential} if there is an essential obstacle enumeration of $G$. In this section, we will prove that essential obstacles are precisely the minimal forbidden circular-arc subgraphs for the class of Helly circular-arc graphs. 

We first observe that shortcuts are not possible in obstacle enumerations of minimal obstacles.

\begin{lem}\label{lem:shortcut} Let $G$ be an obstacle and let $\mathcal Q$ be an obstacle enumeration of $G$. If $\mathcal Q$ has a shortcut, then $G$ is not a minimal obstacle.\end{lem}
\begin{proof} Suppose first that $y_1y_2$ is an inner shortcut of $\mathcal Q$, where $(y_1,y_2)$ is an inner-shortcut pair. Let $\mathcal Q=v_1,v_2,\ldots,v_k$ so that $r(y_2)=v_1$. Let $j\in[k]$ such that $\ell(y_1)=v_j$. As $\ell(y_1)$ is not consecutive to $r(y_2)$ in $\mathcal Q$, $j\geq 3$ and $j<k$. Hence, $\mathcal Q'=v_1,v_2,\ldots,v_j$ is an obstacle enumeration of some induced subgraph of $G-\{v_{j+1},\ldots,v_k\}$, with the same witness enumeration as in $\mathcal Q$ for each edge $v_iv_{i+1}$ such that $i\in[j-1]$ and the witness enumeration $y_1,y_2$ for the edge $v_jv_1$. This proves that some induced subgraph of $G-\{v_{j+1},\ldots,v_k\}$ is an obstacle. In particular, $G$ is not a minimal obstacle.

Suppose now that $y_1y_2$ is an outer shortcut of $\mathcal Q$. Let $\mathcal Q=v_1,v_2,\ldots,v_k$ so that $\overline N_G(y_1)\cap Q=\{v_1\}$ and $\overline N_G(y_2)\cap Q=\{v_2\}$. Let $\mathcal W_1$ be the witness enumeration of the edge $v_1v_2$ in $\mathcal Q$. As $y_1$ and $y_2$ do not occur together in any witness enumeration of $\mathcal Q$, there is at least one vertex $y\in V(\mathcal W_1)-\{y_1,y_2\}$. (Notice that if $\mathcal W_1$ consists of just one vertex, then this vertex has two nonneighbors in $Q$ and, in particular, is different from $y_1$ and $y_2$.) We replace $\mathcal W_1$ by $\mathcal W'_1=y_1,y_2$. Suppose, for a contradiction, that $y$ remains a witness of $\mathcal Q$. Thus, either (i) the witness enumeration of $v_kv_1$ is $u_k,y$ for some vertex $u_k$, or (ii) the witness enumeration of $v_2v_3$ is $y,z_2$ for some vertex $z_2$. However, if (i) holds, then necessarily $\mathcal W_1=y,z_1$ for some vertex $z_1$ and $y$ would be the only witness of the original $\mathcal Q$ such that $\overline N_G(y)=\{v_1\}$, contradicting $y\neq y_1$. Similarly, if (ii) holds, then necessarily $\mathcal W_1=u_1,y$ for some vertex $u_1$ and $y$ would be the only witness of the original $\mathcal Q$ such that $\overline N_G(y)=\{v_2\}$, contradicting $y\neq y_2$. These contradictions show that the $G-y$ contains an induced obstacle. In particular, $G$ is not a minimal obstacle.\end{proof}

Our next result deals with obstacles enumerations having a cover.

\begin{lem}\label{lem:cover} Let $G$ be an obstacle and let $\mathcal Q$ be an obstacle enumeration of $\mathcal Q$. If $\mathcal Q$ has a cover, then either $G$ contains an induced essential obstacle or $G$ contains one of the following minimal forbidden induced subgraphs for the class of Helly circular-arc graphs as an induced subgraph: $C_4^*$, $K_{2,3}$, domino, $G_3$, $\overline{C_6}$, and $\overline{C_5+K_2}$.\end{lem}
\begin{proof} Suppose that the edge $y_1y_2$ is a cover of $\mathcal Q$. Let $\mathcal Q=v_1,v_2,\ldots,v_k$ and let $Q=V(\mathcal Q)$. As $y_1y_2$ is a cover, $\overline N_G(y_1)\cup\overline N_G(y_2)\supseteq Q$. Hence, $k\in\{3,4\}$ and, without loss of generality, $\overline N_G(y_1)\cap Q=\{v_2,v_3\}$. All along this proof, we will refer to vertex $y_1$ as $w_2$.

We consider all possible cases up to symmetry.
\begin{enumerate}[{Case }1:]
 \item\label{case1} \emph{$k=3$ and $\overline N_G(y_2)\cap Q=\{v_3,v_1\}$.} All along this case, we refer to $y_2$ as $w_3$. If the witness enumeration of $v_1v_2$ is $w_1$ for some vertex $w_1$, then $\{v_1,w_2,w_3,v_2,w_1\}$ induces $C_4^*$ or $\{v_1,v_2,v_3,w_1,w_2,w_3\}$ induces $G_3$ or $\overline{C_6}$ in $G$ depending on the number of edges of the subgraph of $G$ induced by $\{w_1,w_2,w_3\}$. Hence, we assume, without loss of generality, that the witness enumeration of $v_1v_2$ is $u_1,z_1$ for some vertices $u_1$ and $z_1$. Thus, $\{u_1,z_1,v_1,v_2,w_3,w_2\}$ induces a pseudo-domino $D$ in $G$. Because of Lemma~\ref{lem:pseudo-domino}, either $G$ contains an induced $K_{2,3}$, domino, $G_3$, or $\overline{C_6}$, or $u_1w_3$, $z_1w_2$, and at least one of the unordered pairs $u_1w_2$ and $z_1w_3$ are edges of $G$. Hence, we assume, without loss of generality, that $u_1w_3$ and $z_1w_2$ are edges of $G$ and, by symmetry, $u_1w_2$ is also an edge of $G$. Therefore, $v_3,u_1,z_1$ is an obstacle enumeration of $G$ with witness enumerations $w_2,v_1$ and $v_1,v_2$ for the edges $v_3u_1$ and $u_1z_1$, respectively, and witness enumeration $v_2,w_3$ or $w_3$ for the edge $z_1v_3$ depending on whether or not $w_3$ is adjacent to $z_1$ in $G$, respectively.
 
 \item \emph{$k=3$ and $\overline N_G(y_2)\cap Q=\{v_1\}$.} By symmetry, we assume, without loss of generality, that the witness enumeration of $v_1v_2$ is $u_1,z_1$, where $y_2=u_1$ and $z_1$ is some vertex. Without loss of generality, $w_2$ is adjacent to $z_1$, since otherwise $\{u_1,z_1,v_1,v_2,w_2\}$ induces $K_{2,3}$ in $G$. Suppose first that the witness enumeration of $v_3v_1$ is $w_3$ for some vertex $w_3$. If $w_2$ were adjacent to $w_3$, then we are in Case~\ref{case1}. Thus, without loss of generality, we assume that $w_2$ is nonadjacent to $w_3$. 
 Also without loss of generality, $u_1w_3$ is an edge of $G$, since otherwise $\{v_1,v_3,u_1,w_2,w_3\}$ induces $C_4^*$ in $G$. It turns out that the obstacle enumeration in the last sentence of Case~\ref{case1} (including the same witness enumerations) is an essential obstacle enumeration of $G$. Suppose now that the witness enumeration of $v_3v_1$ is $u_3,z_3$ for some vertices $u_3$ and $z_3$ (eventually $z_3=u_1$). If $u_3$ is nonadjacent to $u_1$, then the obstacle enumeration $z_1,v_1,v_3$ with witnesses enumerations $v_2,u_1$ and $u_1,w_2$ for the edges $z_1v_1$ and $v_1v_3$, respectively, and witness enumeration $u_3,v_2$ or $u_3$ depending on whether or not $u_3$ is adjacent to $z_1$, respectively, is an essential obstacle enumeration of the subgraph of $G$ induced by $\{v_1,v_2,v_3,u_1,z_1,w_2,u_3\}$. Hence, we assume, without loss of generality, that $u_3$ is adjacent to $u_1$. Moreover, without loss of generality, $u_3$ is adjacent to $w_2$, since otherwise $\{v_1,v_3,u_1,w_2,u_3\}$ induces $K_{2,3}$ in $G$. 
 Furthermore, without loss of generality, $u_3$ is adjacent to $z_1$, since otherwise $\{w_2,u_3,v_2,v_3,z_1,v_1,u_1\}$ induces $\overline{C_5+K_2}$ in $G$. Thus, the obstacle enumeration $v_1,v_2,v_3$ with witness enumerations $u_1,z_1$, $z_1,u_3$, and $u_3,u_1$ for the edges $v_1v_2$, $v_2v_3$, and $v_3v_1$, respectively, is an essential obstacle enumeration of the subgraph of $G$ induced by $\{v_1,v_2,v_3,u_1,z_1,u_3\}$. 

 \item\label{case3} \emph{$k\geq 4$}. As $\overline N_G(y_1)\cup\overline N_G(y_2)\supseteq Q$, necessarily $k=4$ and $\overline N_G(y_2)\cap Q=\{v_4,v_1\}$. We refer to $y_2$ as $w_4$. Suppose first that the witness enumeration of $v_1v_2$ in $\mathcal Q$ is $u_1,z_1$ for some vertices $u_1$ and $z_1$. Thus, $\{u_1,z_1,v_1,v_2,w_4,w_2\}$ induces a pseudo-domino in $G$. By virtue of Lemma~\ref{lem:pseudo-domino}, we assume, without loss of generality, that $u_1w_4$ and $z_1w_2$ are edges of $G$ and, by symmetry, $w_2u_1$ is also an edge of $G$. On the one hand, if $w_4$ is nonadjacent to $z_1$, then the obstacle enumeration $u_1,z_1,v_3$ with witness enumerations $v_1,v_2$, $w_4,w_2$, and $w_2,v_1$ for the edges $u_1z_1$, $z_1v_3$, and $v_3u_1$, respectively, is an essential obstacle enumeration of the subgraph of $G$ induced by $\{v_1,v_2,v_3,u_1,z_1,w_2,w_4\}$. On the other hand, if $w_4$ is adjacent to $z_1$, then the obstacle enumeration $u_1,v_3,v_4,z_1$ with witness enumerations $v_1,w_2$, $w_2,w_4$, $w_4,v_2$, and $v_2,v_1$ for the edges $u_1u_3$, $v_3v_4$, $v_4z_1$, and $z_1u_1$, respectively, is an essential obstacle enumeration of the subgraph of $G$ induced by $\{v_1,v_2,v_3,v_4,u_1,z_1,w_2,w_4\}$. This completes the proof in case the witness enumeration of $v_1v_2$ in $\mathcal Q$ is $u_1,z_1$ for some vertices $u_1$ and $z_1$. Hence, we assume, without loss of generality, that the witness enumeration of $v_1v_2$ in $\mathcal Q$ is $w_1$ for some vertex $w_1$. Without loss of generality, $w_1$ is adjacent to at least one of $w_2$ and $w_4$ since otherwise $\{v_1,v_2,w_2,w_4,w_1\}$ induces $C_4^*$ in $G$. If $w_1$ is adjacent to $w_2$ but nonadjacent to $w_4$, then $\{v_1,v_3,w_1,w_2,w_4\}$ induces $K_{2,3}$ in $G$. Symmetrically, if $w_1$ is adjacent to $w_4$ but nonadjacent to $w_2$, then $\{v_2,v_4,w_1,w_2,w_4\}$ induces $K_{2,3}$ in $G$. Finally, if $w_1$ is adjacent to both $w_2$ and $w_4$, then the obstacle enumeration $w_1,v_3,v_4$ with witness enumerations $v_1,w_2$, $w_2,w_4$, and $w_4,v_2$ for the edges $w_1v_3$, $v_3v_4$, and $v_4w_1$, respectively, is an essential obstacle enumeration of the subgraph of $G$ induced by $\{v_1,v_2,v_3,v_4,w_1,w_2,w_4\}$.\qedhere
\end{enumerate}\end{proof}

We now show that shortcuts and covers are the only faults that prevent an obstacle enumeration from being essential.

\begin{lem}\label{lem:noshort+nocover=essential} Let $G$ be an obstacle and let $\mathcal Q$ be an obstacle enumeration of $G$. If an edge $e$ of $G$ joining two witnesses of $\mathcal Q$ is neither a shortcut nor a cover, then $e$ is valid. Therefore, if $\mathcal Q$ has no shortcut and no cover, then $\mathcal Q$ is an essential obstacle enumeration and $G$ is an essential obstacle.\end{lem}
\begin{proof} Let $y_1$ and $y_2$ be two adjacent witnesses of $\mathcal Q$ such that the edge $y_1y_2$ is not a shortcut. We will prove that either $y_1y_2$ is valid or is a cover. Let $Q=V(\mathcal Q)$.

As $y_1y_2$ is not a shortcut, all the following statements holds:
\begin{enumerate}[(i)]
 \item $\ell(y_1)\in\overline N_G(y_2)$, $r(y_2)\in\overline N_G(y_1)$, or $\ell(y_1)$ and $r(y_2)$ are consecutive in $\mathcal Q$.
 \item $\ell(y_2)\in\overline N_G(y_1)$, $r(y_1)\in\overline N_G(y_2)$, or $\ell(y_2)$ and $r(y_1)$ are consecutive in $\mathcal Q$.
 \item\label{it:noshortcut3} If $\ell(y_1)=r(y_1)$, $\ell(y_2)=r(y_2)$, $\ell(y_1)$ and $\ell(y_2)$ are consecutive in $\mathcal Q$, then $y_1$ and $y_2$ occur together in some witness enumeration of $\mathcal Q$.
\end{enumerate}

If $\ell(y_1)\in\overline N_G(y_2)$, there are three possible cases:
\begin{enumerate}[{Case }1:]
 \item \emph{$\ell(y_2)\in\overline N_G(y_1)$ holds.} Hence, either $\ell(y_1)=\ell(y_2)$ or $\overline N_G(y_1)\cup\overline N_G(y_2)\supseteq Q$. In the former case, $y_1y_2$ is valid because $\overline N_G(y_1)\cap Q$ and $\overline N_G(y_2)\cap Q$ are comparable, whereas in the latter case, $y_1y_2$ is a cover.
 
 \item \emph{$r(y_1)\in\overline N_G(y_2)$ holds.} Thus, $\overline N_G(q_1)\cap Q\subseteq\overline N_G(q_2)\cap Q$ and $y_1y_2$ is valid.

 \item \emph{$\ell(y_2)\notin\overline N_G(y_1)$, $r(y_1)\notin\overline N_G(y_2)$, and $\ell(y_2)$ and $r(y_1)$ are consecutive in $\mathcal Q$.} Suppose, for a contradiction, that $\ell(y_2)=v_i$ and $r(y_1)=v_{i+1}$ for some $i\in[k]$. Since $\ell(y_2)\notin\overline N_G(y_1)$ and $r(y_1)\notin\overline N_G(y_2)$, $\ell(y_1)=r(y_1)\neq\ell(y_2)=r(y_2)$, which contradicts $\ell(y_1)\in\overline N_G(y_2)$. This contradiction proves that $r(y_1)=v_i$ and $\ell(y_2)=v_{i+1}$ for some $i\in[k]$. As $\ell(y_1)\in\overline N_G(y_2)$ but $\ell(y_2)\notin\overline N_G(y_1)$, $r(y_2)=\ell(y_1)=v_{i+2}$. Since $r(y_1)\notin\overline N_G(y_2)$, $r(y_1)=v_{i+3}$. From $v_i=v_{i+3}$, we conclude that $k=3$ and $\overline N_G(y_1)\cup\overline N_G(y_2)\supseteq\{v_{i+1},v_{i+2}\}\cup\{v_{i+2},v_i\}=\{v_i,v_{i+1},v_{i+2}\}=Q$.
\end{enumerate}

The cases where $r(y_2)\in\overline N_G(y_1)$, $\ell(y_2)\in\overline N_G(y_1)$, or $r(y_1)\in\overline N_G(y_2)$ are symmetric to the case $\ell(y_1)\in\overline N_G(y_2)$ discussed above. Hence, in order to complete the proof of the lemma, it suffices to consider the case where $\ell(y_1)\notin\overline N_G(y_2)$, $r(y_2)\notin\overline N_G(y_1)$, $\ell(y_2)\notin\overline N_G(y_1)$, $r(y_1)\notin\overline N_G(y_2)$, $\ell(y_1)$ and $r(y_2)$ are consecutive in $\mathcal Q$, and $\ell(y_2)$ and $r(y_1)$ are consecutive in $\mathcal Q$. If $\ell(y_1)$ is immediately after $r(y_2)$ in $\mathcal Q$ and $\ell(y_2)$ is immediately after $r(y_1)$ in $\mathcal Q$, then $y_1y_2$ is a cover. Thus, without loss of generality, $r(y_2)$ is immediately after $\ell(y_1)$ in $\mathcal Q$. As $r(y_1)\notin\overline N_G(y_2)$, $r(y_1)=\ell(y_1)$. Symmetrically, as $\ell(y_2)\notin\overline N_G(y_1)$, $\ell(y_2)=r(y_2)$. Because of \eqref{it:noshortcut3}, $y_1$ and $y_2$ occur together in some witness enumeration of $\mathcal Q$ and, by definition, $y_1y_2$ is valid.\end{proof}

Based on the preceding three lemmas, we now show that, given a graph $G$ with an obstacle enumeration in it, it is possible to find in linear time an essential obstacle or one of the six graphs in Figure~\ref{fig:forb-HCA} contained in $G$ as an induced subgraph.

\begin{thm}\label{thm:essential} Given a graph $G$ and an obstacle enumeration $\mathcal Q$ in $G$, it is possible to find in linear time either an essential obstacle enumeration of some induced subgraph of $G$ or an induced subgraph of $G$ isomorphic to $C_4^*$, $K_{2,3}$, domino, $G_3$, $\overline{C_6}$, or $\overline{C_5+K_2}$. Moreover, if $G$ is a circular-arc graph, given a circular-arc model of $G$ and an obstacle enumeration $\mathcal Q$ in $G$, an essential obstacle enumeration of some induced subgraph of $G$ can be found in $O(n)$ time.\end{thm}
\begin{proof} We visit each of the edges joining two current witnesses of $\mathcal Q$, while updating $\mathcal Q$ (including its witness enumerations, the list of nonneighbors in $\mathcal Q$ for each witness of $\mathcal Q$, and the witness enumerations of $\mathcal Q$ in which each vertex occurs), as follows. More precisely, we visit the edges $y_1y_2$ joining two current witnesses $y_1$ and $y_2$ of the current $\mathcal Q$ doing the following:
\begin{itemize}
 \item If $y_1y_2$ is a cover of $\mathcal Q$, then, proceeding as in the proof of Lemma~\ref{lem:cover}, we output either an essential obstacle enumeration of some induced subgraph of $G$ or an induced subgraph of $G$ isomorphic to $C_4^*$, $K_{2,3}$, domino, $G_3$, $\overline{C_6}$, or $\overline{C_5+K_2}$, and stop.
 
 \item If $y_1y_2$ is a shortcut of $\mathcal Q$, then we modify $\mathcal Q$ as in the proof of Lemma~\ref{lem:shortcut}. We call this a \emph{shrinking operation} as it decreases by at least one the number of vertices of the induced subgraph of $G$ of which $\mathcal Q$ is an obstacle enumeration. Notice that, after the shrinking operation, the edge $y_1y_2$ is valid for the resulting $\mathcal Q$.
 
 \item If $y_1y_2$ is valid, we do not modify $\mathcal Q$.
\end{itemize}
Because of Lemma~\ref{lem:noshort+nocover=essential}, one of the above three cases occurs. By definition, if an edge $y_1y_2$ is found valid for the current $\mathcal Q$, then it cannot become a shortcut or a cover of $\mathcal Q$ after any number of shrinking operations. (Eventually, $y_1y_2$ will stop being valid if at least one of its endpoints is no longer a witness of $\mathcal Q$.) Hence, if after having visited all the edges $y_1y_2$, we have found no cover, then, by Lemma~\ref{lem:noshort+nocover=essential}, the final $\mathcal Q$ is an essential obstacle enumeration of some induced subgraph of $G$. As performing all the shrinking operations takes $O(n)$ time in total and any obstacle enumeration having a cover involves at most ten vertices (meaning those in the core plus the witnesses), the whole procedure can be completed in linear time.

For the analysis when $G$ is given through one of its circular-arc models, we introduce some definitions. Let $\mathcal Q=v_1,v_2,\ldots,v_k$ be an obstacle enumeration. We call \emph{surrounding edges of $\mathcal Q$} to the edges $v_iv_{i+1}$ for every $i\in[k]$. If $y_1y_2$ is a valid edge of $\mathcal Q$, we define the \emph{support of $y_1y_2$ in $\mathcal Q$} as follows. If $y_1$ and $y_2$ occur together in the witness enumeration of the edge $v_iv_{i+1}$ for some $i\in[k]$, then the support of $y_1y_2$ is defined to be the edge $v_iv_{i+1}$. If $y_1$ and $y_2$ do not occur together in any witness enumeration of $\mathcal Q$, then the support of $y_1y_2$ is the unique vertex in the singleton $\overline N_G(y_1)\cap\overline N_G(y_2)$. Roughly speaking, assuming the witnesses of $\mathcal Q$ are labeled as in the definition of obstacles, if $y_1y_2$ equals $u_iz_i$ for some $i\in[k]$, then the support of $y_1y_2$ is the edge $v_iv_{i+1}$, whereas, if $y_1y_2$ equals $z_{i-1}w_i$, $w_{i-1}u_i$, or $z_{i-1}u_i$ for some $i\in[k]$, then the support of $y_1y_2$ is the vertex $v_i$.

Suppose now that instead of the graph $G$, a circular-arc model $\mathcal A$ of $G$ is given as input. We apply the same procedure described at the beginning of this proof but visiting the edges $y_1y_2$ joining two current witnesses of the current $\mathcal Q$ as found when traversing $\mathcal A$. As none of the graphs $C_4^*$, $K_{2,3}$, domino, $G_3$, $\overline{C_6}$, or $\overline{C_5+K_2}$ is a circular-arc graph, the output will be some essential obstacle enumeration of some induced subgraph of $G$. If a cover of $\mathcal Q$ is visited, then the induced subgraph of $G$ of which $\mathcal Q$ is an obstacle enumeration has at most ten vertices and, consequently, after the circular-arc submodel of $\mathcal A$ corresponding to the arcs representing these at most ten vertices is extracted in $O(n)$ time, the desired essential obstacle enumeration can be found in additional $O(1)$ time. As each time a shortcut edge is visited, the number of vertices of the induced subgraph of $G$ of which $\mathcal Q$ is an obstacle enumeration decreases by at least one, the number of shortcut edges visited all along the execution of the algorithm is $O(n)$. We call a \emph{surrounding edge of $G$} to any edge of $G$ which is a surrounding edge of any of the different obstacle enumerations $\mathcal Q$ all along the execution of the algorithm. As the number of surrounding edges of $G$ increases by at most one each time a shortcut edge is visited and remains the same when a valid edge is visited, the total number of surrounding edges of $G$ is $O(n)$. Noticing that: (1) each edge found valid during the execution of the algorithm has as support either a vertex or a surrounding edge of $G$, (2) each vertex can serve as support to at most one valid edge all along the execution of the algorithm, and (3) each surrounding edge of $G$ can serve as support to at most four different edges all along the execution of the algorithm, we conclude that the total number of edges found valid all along the execution is $O(n)$. Hence, the total number of edges visited all along the execution is $O(n)$. As performing all the shrinking operations takes $O(n)$ time in total, the $O(n)$ time bound for the whole procedure follows.\end{proof}

We now prove that each essential obstacle is a minimal forbidden induced circular-arc subgraph for the class of Helly circular-arc graphs.

\begin{lem}\label{lem:essential-HCA} Every essential obstacle is a circular-arc graph and a minimal forbidden induced subgraph for the class of Helly circular-arc graphs.\end{lem}
\begin{proof} Let $G$ be an essential obstacle and let $\mathcal Q=v_1,v_2,\ldots,v_k$ be an essential obstacle enumeration of $G$. We denote by $\mathcal W_i$ the witness enumeration of the edge $v_iv_{i+1}$ in $\mathcal Q$ for each $i\in[k]$. All along this proof, all subindices on $\ell_i$, $r_i$, and $m_i$ are modulo $k$. Let $Q=V(\mathcal Q)$.

As $\mathcal Q$ is essential, if there is some maximal clique of $G$ consisting only of witnesses of $\mathcal Q$, then necessarily $k=3$ and there are three vertices $u_1$, $u_2$, and $u_3$ of $G$ such that $\mathcal W_1=u_1,u_2$, $\mathcal W_2=u_2,u_3$, and $\mathcal W_3=u_3,u_1$. In such a case, $G$ is isomorphic to $\overline{3K_2}$ and it can be verified by inspection that $G$ is a circular-arc graph and a minimal forbidden induced subgraph for the class of Helly circular-arc graphs. Henceforth, we assume, without loss of generality, that $G$ is not isomorphic to $\overline{3K_2}$ and, consequently, every maximal clique of $G$ has at least one vertex in the set $Q$.

We build a circular-arc model of $G$ as follows. Let $C$ be a circle. Given two points $p$ and $q$ of $C$, we denote by $(p,q)$ the open arc of points of $C$ found when traversing $C$ in clockwise direction from $p$ to $q$; the semi-open arcs $(p,q]$ and $[p,q)$ and the closed arc $[p,q]$ are defined similarly. Let $\ell_1,\ell_2,\ldots,\ell_k$ be $k$ different points of $C$ occurring in that precise order when traversing $C$ in clockwise direction. Let $r_1,r_2,\ldots,r_k$ be other $k$ points of $C$ such that $r_i$ belongs to the arc $(\ell_{i+1},\ell_{i+2})$ for each $i\in[k]$. Finally, let $m_1,m_2,\ldots,m_k$ be other $k$ points of $C$ such that $m_i$ belongs to the arc $(r_{i-1},\ell_{i+1})$ for each $i\in[k]$. We define an arc $A_v$ for each vertex $v$ of $G$ as follows:
\begin{itemize}
 \item If $v=v_i$ for some $i\in[k]$, $A_{v_i}=(r_i,\ell_i)$ (i.e., $A_{v_i}=C-[\ell_i,r_i]$).

 \item If $v=w_i$ for some $i\in[k]$ such that $W_i=w_i$, then $A_{w_i}=[\ell_{i+1},r_i]$.
 
 \item If $v=u_i$ for some $i\in[k]$ such that $W_i=u_i,z_i$, then
 \[
    A_{u_i}=\begin{cases}
               [r_{i-1},r_i]&\text{if $\mathcal W_{i-1}=w_{i-1}$ and $w_{i-1}$ is adjacent to $u_i$},\\
               [\ell_i,r_i]&\text{if $\mathcal W_{i-1}=u_{i-1},z_{i-1}$ and $z_{i-1}$ and is equal to $u_i$},\\
               [m_i,r_i]&\text{if $\mathcal W_{i-1}=u_{i-1},z_{i-1}$ and $z_{i-1}$ is adjacent to $u_i$},\\
               (m_i,r_i]&\text{otherwise.}
               \end{cases}
 \]

 \item If $v=z_i$ for some $i\in[k]$ such that $\mathcal W_i=u_i,z_i$, then
 \[
    A_{z_i}=\begin{cases}
              [\ell_{i+1},\ell_{i+2}]&\text{if $\mathcal W_{i+1}=w_{i+1}$ and $w_{i+1}$ is adjacent to $z_i$},\\
              [\ell_{i+1},r_{i+1}]&\text{if $\mathcal W_{i+1}=u_{i+1},z_{i+1}$ and $u_{i+1}$ is equal to $z_i$},\\
              [\ell_{i+1},m_{i+1}]&\text{if $\mathcal W_{i+1}=u_{i+1},z_{i+1}$ and $u_{i+1}$ adjacent to $z_i$},\\
              [\ell_{i+1},m_{i+1})&\text{otherwise.}
              \end{cases}
 \]
\end{itemize}
It is easy to verify that $\mathcal A=\{A_v:v\in V(G)\}$ is a circular-arc model of $G$. We list all maximal cliques of $G$ different from $Q$ and give a clique point in $\mathcal A$ for each of them:
\begin{itemize}
 \item $\{u_i,z_i\}\cup(Q-\{v_i,v_{i+1}\})$ for each $i\in[k]$ such that $\mathcal W_i=u_i,z_i$. For each such $i$, any point in the arc $(\ell_{i+1},r_i)$ is a clique point for this clique.
 
 \item $\{w_i,u_{i+1}\}\cup(Q-\{v_i,v_{i+1}\})$ for each $i\in[k]$ such that $\mathcal W_i=w_i$, $\mathcal W_{i+1}=u_{i+1},z_{i+1}$, and $w_iu_{i+1}\in E(G)$. For each such $i$, $r_i$ is a clique point for this clique.
 
 \item $\{z_{i-1},w_i\}\cup(Q-\{v_i,v_{i+1}\})$ for each $i\in[k]$ such that $\mathcal W_i=w_i$, $\mathcal W_{i-1}=u_{i-1},z_{i-1}$, and $w_iz_{i-1}\in E(G)$. For each such $i$, $\ell_{i+1}$ is a clique point for this clique.

 \item $\{z_{i-1},u_i\}\cup(Q-\{v_i\})$ for each $i\in[k]$ such that $\mathcal W_{i-1}=u_{i-1},z_{i-1}$, $\mathcal W_i=u_i,z_i$, and $z_{i-1}u_i\in E(G)$. For each such $i$, $m_i$ is a clique point for this clique.
 
 \item $\{w_i\}\cup(Q-\{v_i,v_{i+1}\})$ for each $i\in[k]$ such that $\mathcal W_i=w_i$ and $N_G(w_i)\subseteq Q$. For each such $i$, each point in the arc $(\ell_{i+1},r_i)$ is a clique point for this clique.
 
 \item $\{u_i\}\cup(Q-\{v_i\})\}$ for each $i\in[k]$ such that $\mathcal W_i=u_i,z_i$ unless both $\mathcal W_{i-1}=u_{i-1},z_{i-1}$ and $z_{i-1}u_i\in E(G)$. For each such $i$, any point in the arc $(m_i,\ell_{i+1})$ is a clique point for this clique.
 
 \item $\{z_{i-1}\}\cup(Q-\{v_i\})\}$ for each $i\in[k]$ such that $\mathcal W_{i-1}=z_{i-1},u_{i-1}$ unless both $\mathcal W_i=u_i,z_i$ and $z_{i-1}u_i\in E(G)$. For each such $i$, any point in the arc $(r_{i-1},m_i)$ is a clique point for this clique.
\end{itemize}

Let $j\in[k]$. We claim that, $\mathcal A-A_{v_j}$ is a Helly circular-arc model of $G-v_j$. In order to prove the claim, let $\mathcal A'=\mathcal A-A_{v_j}$ and let $Q,Q_1,\ldots,Q_r$ be the maximal cliques of $G$. Clearly, each maximal clique of $G-v_j$ equals either $Q-v_j$ or $Q_s-v_j$ for some $s\in[r]$. As for each $s\in[r]$, the clique point of $Q_s$ in $\mathcal A$ is a clique point of $Q_s-v_j$ in $\mathcal A'$, it only remains to prove that either $Q-v_j$ is not a maximal clique of $G-v_j$ or there is a clique point for $Q-v_j$ in $\mathcal A'$. If $W_{j-1}=u_{j-1},z_{j-1}$ or $W_j=u_j,z_j$, then $Q-v_j$ is not a maximal clique of $G$ because $G$ has some maximal clique $Q_s$ containing $Q-v_j$ and at least one witness of $\mathcal Q$ and, consequently, $Q_s-v_j$ is a clique of $G-v_j$ properly containing $Q-v_j$. If, on the contrary, $W_{j-1}=w_{j-1}$ and $W_j=w_j$, then, by construction, $m_j$ is a clique point of $Q-v_j$ in $\mathcal A'$. This completes the proof of the claim.

Let $y\in V(G)-Q$. We build a Helly circular-arc model of the graph $G-y$ as follows.
\begin{itemize}
 \item If $y=w_i$, where $\mathcal W_i=w_i$ for some $i\in[k]$, the circular-arc model $\mathcal A'$ that arises from $\mathcal A$ by removing $A_{w_i}$ and replacing $A_{v_i}$ by $(\ell_{i+1},\ell_i)$ and $A_{v_{i+1}}$ by $(r_{i+1},r_i)$ is a Helly circular-arc model for $G-w_i$ because each point in the arc $(\ell_{i+1},r_i)$ is clique point for $Q$ in $\mathcal A'$.
 
 \item If $y=u_i$, where $\mathcal W_i=u_i,z_i$ for some $i\in[k]$, the circular-arc model $\mathcal A'$ that arises from $\mathcal A$ by removing $A_{u_i}$ and replacing $A_{v_i}$ by $(m_i,\ell_i)$ is a Helly circular-arc model of $G-u_i$ because each point in the arc $(m_i,\ell_{i+1})$ is a clique point for $Q$ in $\mathcal A'$.
 
 \item If $y=z_{i-1}$, where $\mathcal W_{i-1}=u_{i-1},z_{i-1}$ for some $i\in[k]$, the circular-arc model $\mathcal A'$ that arises from $\mathcal A$ by removing $A_{z_i}$ and replacing $A_{v_i}$ by $(r_i,m_i)$ is a Helly circular-arc model of $G-z_{i-1}$ because each point in the arc $(r_{i-1},m_i)$ is a clique point for $Q$ in $\mathcal A'$.
\end{itemize}
As $G$ is not a Helly circular-arc graph (because it is an obstacle) but $G-v$ has a Helly circular-arc model for each $v\in V(G)$, $G$ is a minimal forbidden subgraph for the class of Helly circular-arc graph. This completes the proof of the lemma.\end{proof}

We are now ready to prove that the minimal forbidden induced circular-arc subgraphs for the class of Helly circular-arc graphs are precisely the essential obstacles.

\begin{proof}[Proof of Theorem~\ref{thm:main}] Let $G$ be a minimal forbidden induced circular-arc subgraph for the class of Helly circular-arc graphs. As $G$ is a circular-arc graph and not a Helly circular-arc graph, Theorem~\ref{thm:HCAobst} implies that $G$ contains an induced obstacle. As $G$ is a circular-arc graph, Theorem~\ref{thm:essential} implies that $G$ contains some induced essential obstacle $H$. Because of the minimality of $G$, $G$ equals $H$. Conversely, by Lemma~\ref{lem:essential-HCA}, essential obstacles are minimal forbidden induced circular-arc subgraphs for the class of Helly circular-arc graph.\end{proof}

As mentioned in the introduction, the algorithm by Joeris et al.~\cite{MR2765574} produces positive and negative certificates when the input graph is a circular-arc graph.

\begin{thm}[\cite{MR2765574}]\label{thm:joeris} Given a circular-arc graph $G$, it is possible to find in linear time either a Helly circular-arc model of $G$ or an obstacle enumeration of some induced subgraph of $G$. Moreover, if a circular-arc model of $G$ is given as input, the time bound reduces to $O(n)$.\end{thm}

Theorems~\ref{thm:main}, \ref{thm:essential}, and~\ref{thm:joeris} imply the following.

\begin{cor}\label{cor:cAminHCA} Given a circular-arc graph $G$, it is possible to find in linear time either a Helly circular-arc model of $G$ or an essential obstacle enumeration of some minimal forbidden induced subgraph for the class of Helly circular-arc graphs contained in $G$ as an induced subgraph. Moreover, if a circular-arc model of $G$ is given as input, the time bound reduces to $O(n)$.\end{cor}

\begin{rmk} The total number of minimal forbidden induced subgraphs for the class of Helly circular-arc graphs having at most $N$ vertices grows exponentially as $N$ increases. For instance, suppose we want to build an essential obstacle $G$ with essential obstacle enumeration $Q=v_1,\ldots,v_k$ for some $k\geq 3$ such that, for each $i\in[k]$, the witness enumeration of $v_iv_{i+1}$ is $u_i,z_i$ for some vertices $u_i$ and $z_i$. For each $i\in[k]$, we have three choices: (1) $z_{i-1}=u_i$, (2) $z_{i-1}$ and $u_i$ are different and nonadjacent, or (3) $z_{i-1}$ and $u_i$ are adjacent. All these choices can be made independently for each $i$ because any combination of them always leads to an essential obstacle $G$, as long as the only edges in $G$ joining two witnesses of $\mathcal Q$ are precisely those produced by choosing (3) for certain values of $i$. We may associate with $G$ a sequence $a_1,\ldots,a_k$ of values $1$, $2$, and $3$ corresponding to the choices made for each $i$ from $1$ to $k$. Clearly, two such essential obstacles $G_1$ and $G_2$ are isomorphic if and only if their corresponding sequences belong to the same equivalence class of sequences of length $k$ with values $1$, $2$, and $3$, up to rotations and reversals. These equivalence classes are known as \emph{ternary bracelets of length $k$}. Hence, there are as many such nonisomorphic essential obstacles as the number of ternary bracelets of length $k$, which is known to be
\begin{equation}\label{bracelets}
  \frac 1{2k}\sum_{d\mid k}\varphi(d)3^{k/d}+
   \begin{cases}
    3^{k/2} &\text{if $k$ is even,}\\
    \frac 12 3^{(k+1)/2} &\text{if $k$ is odd,}     
   \end{cases} \end{equation}
where $d\mid k$ means `$d$ is a positive divisor of $k$' and $\varphi$ denotes Euler's totient function. (For a derivation of \eqref{bracelets}, see e.g.~\cite{MR0096594}.)\end{rmk}

\section{Helly circular-arc graphs with no claw and no 5-wheel}\label{sec:quasi-line}

In this section, we give the minimal forbidden induced subgraph characterization of Helly circular-arc graphs restricted to graphs containing no induced claw and no induced $5$-wheel. Moreover, we show that in linear time it is possible to find an induced claw, an induced $5$-wheel, or an induced minimal forbidden induced subgraph for the class of Helly circular-arc graphs, in any given graph that is not a Helly circular-arc graph. Some small graphs needed in what follows are depicted in Figure~\ref{fig:more-small}.

\begin{figure}[t!]
\ffigbox[\textwidth]{%
\ffigbox[\FBwidth]{%
\begin{subfloatrow}
\ffigbox[0.18\textwidth]{\includegraphics{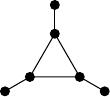}}{\caption{net}}
\ffigbox[0.18\textwidth]{\includegraphics{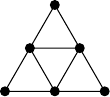}}{\caption{tent}}
\ffigbox[0.18\textwidth]{\includegraphics{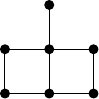}}{\caption{$H_2$}}
\ffigbox[0.18\textwidth]{\includegraphics{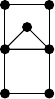}}{\caption{$H_3$}}
\ffigbox[0.18\textwidth]{\includegraphics{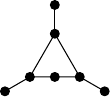}}{\caption{$H_4$}}
\end{subfloatrow}}{}
\bigskip
\ffigbox[\FBwidth]{%
\begin{subfloatrow}
\ffigbox[0.18\textwidth]{\includegraphics{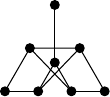}}{\caption{$F_1$}}
\ffigbox[0.18\textwidth]{\includegraphics{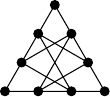}}{\caption{$F_2$}}
\ffigbox[0.18\textwidth]{\includegraphics{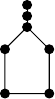}}{\caption{$F_3$}}
\ffigbox[0.18\textwidth]{\includegraphics{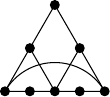}}{\caption{$F_4$}}
\ffigbox[0.18\textwidth]{\includegraphics{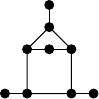}}{\caption{$F_5$}}
\end{subfloatrow}}{}
\bigskip
\ffigbox[\FBwidth]{%
\begin{subfloatrow}
\ffigbox[0.18\textwidth]{\includegraphics{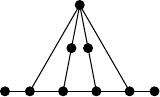}}{\caption{$F_6$}}
\ffigbox[0.18\textwidth]{\includegraphics{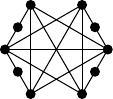}}{\caption{$F_7$}}
\ffigbox[0.18\textwidth]{\includegraphics{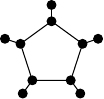}}{\caption{$F_8$}}
\ffigbox[0.18\textwidth]{\includegraphics{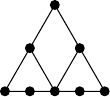}}{\caption{$\BII{2}$}}
\ffigbox[0.18\textwidth]{\includegraphics{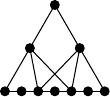}}{\caption{$\BII{3}$}}
\end{subfloatrow}}{}
\bigskip
\ffigbox[\FBwidth]{%
\begin{subfloatrow}
\ffigbox[0.18\textwidth]{\includegraphics{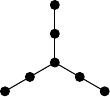}}{\caption{$\BIII{1}$}}
\ffigbox[0.18\textwidth]{\includegraphics{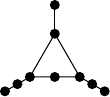}}{\caption{$\BIII{1}$}}
\ffigbox[0.18\textwidth]{\includegraphics{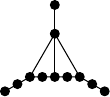}}{\caption{$\BIII{2}$}}
\ffigbox[0.18\textwidth]{\includegraphics{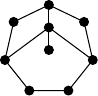}}{\caption{$Z$}}
\end{subfloatrow}}{}
}{\caption{Some small graphs}\label{fig:more-small}}
\end{figure}

We begin by determining all claw-free essential obstacles.

\begin{lem}\label{lem:claw-free-obst} The claw-free essential obstacles are $\overline{3K_2}$, $\overline{P_7}$, $\overline{F_1}$, $\overline{F_2}$, $\overline{F_3}$, $\overline{F_4}$, $\overline{H_3}$, net, $\overline{2P_4}$, $\overline{F_5}$, $\overline{F_6}$, $\overline{F_7}$, and $\overline{F_8}$.\end{lem}
\begin{proof} Let $G$ be an essential obstacle. Let $\mathcal Q=v_1,v_2,\ldots,v_k$ be an essential obstacle enumeration of $G$ and, for each $i\in[k]$, let $\mathcal W_i$ be the witness enumeration of $v_iv_{i+1}$ in $\mathcal Q$.

For each $i\in[k]$, we have the following facts:
\begin{enumerate}[{Fact }1:]
 \item\label{fact1} \emph{If $\mathcal W_{i-1}=u_{i-1},z_{i-1}$ and $\mathcal W_i=u_i,z_i$, then $z_{i-1}$ and $u_i$ are either equal or adjacent.} Otherwise, $\{v_{i-1},z_{i-1},v_i,u_i\}$ induces claw in $G$.

 \item \emph{If $\mathcal W_{i-1}=u_{i-1},z_{i-1}$ and $\mathcal W_i=w_i$, then $z_{i-1}$ and $w_i$ are adjacent.} Otherwise, $\{v_{i-1},z_{i-1},v_i,w_i\}$ induces claw in $G$.

 \item \emph{If $\mathcal W_{i-1}=w_{i-1}$ and $\mathcal W_i=u_i,z_i$, then $w_{i-1}$ and $u_i$ are adjacent.} Otherwise, $\{v_{i+1},\linebreak w_{i-1},v_i,u_i\}$ induces claw in $G$.

 \item\label{fact4} \emph{If $k\geq 4$, then it is not possible that $\mathcal W_{i-1}=w_{i-1}$ and $\mathcal W_i=w_i$ simultaneously.} Otherwise, $\{v_{i+2},w_{i-1},v_i,w_i\}$ induces claw in $G$. (Notice that $w_{i-1}$ and $w_i$ are nonadjacent because $\mathcal Q$ is essential.)
 
 \item\label{fact5} \emph{If $k\geq 4$, $\mathcal W_{i-1}=u_{i-1},z_{i-1}$, and $\mathcal W_i=u_i,z_i$, then $z_{i-1}=u_i$ unless $G$ is isomorphic to $\overline{F_5}$.} Suppose $k\geq 4$, $\mathcal W_{i-1}=u_{i-1},z_{i-1}$, $\mathcal W_i=u_i,z_i$, and $z_{i-1}\neq u_i$. Because of Fact~\ref{fact1}, $z_{i-1}$ is adjacent to $u_i$. Because of Fact~\ref{fact4} and by symmetry, we assume, without loss of generality, that $\mathcal W_{i-2}=u_{i-2},z_{i-2}$. As $\mathcal Q$ is essential and $k\geq 4$, $z_{i-1}$ is nonadjacent to $u_{i-2}$ and to $z_i$. Thus, $u_{i-2}$ is adjacent to $z_i$, since otherwise $\{v_{i-1},u_{i-2},z_{i-1},z_i\}$ would induce claw in $G$. As $\mathcal Q$ is essential, $k=4$ and $\mathcal W_{i+1}=z_i,u_{i-2}$. Also because $\mathcal Q$ is essential and $k\geq 4$, $u_i$ is nonadjacent to $u_{i-1}$ and $u_{i-2}$. Hence, $u_{i-2}$ is adjacent to $u_{i-1}$, since otherwise $\{v_{i+1},u_{i-2},u_{i-1},u_i\}$ would induced claw in $G$. The essentiality of $\mathcal Q$ implies $z_{i-2}=u_{i-1}$. We conclude that $G$ is isomorphic to $\overline{F_5}$.
 
 \item \emph{If $k\geq 5$, then $G$ is isomorphic to $\overline{F_8}$.} Suppose, for a contradiction, that there exists no $i\in[k]$ such that $\mathcal W_{i-1}=u_{i-1},z_{i-1}$ and $\mathcal W_i=u_i,z_i$. Because of Fact~\ref{fact4}, $k$ is even and, without loss of generality, for each $i\in[k]$, $\mathcal W_j=u_i,z_i$ if $i$ is odd and $\mathcal W_i=w_i$ if $i$ is even. In particular, $\{v_2,u_1,u_3,u_5\}$ induces claw in $G$. This contradiction proves that there is some $i\in[k]$ such that $\mathcal W_{i-1}=u_{i-1},z_{i-1}$ and $\mathcal W_i=u_i,z_i$. Because of Fact~\ref{fact5}, $z_{i-1}=u_i$. If there were some witness $y\in V(\mathcal W_{i+2})$ simultaneously nonadjacent to $u_{i-1}$ and $z_i$, then $\{v_i,u_{i-1},z_i,y\}$ would induce claw in $G$. Hence, $k=5$ and $\mathcal W_{i+2}=u_{i+2},v_{i+2}$, where $u_{i+2}$ is adjacent to $z_i$ and $z_{i+2}$ is adjacent to $u_{i-1}$. As $\mathcal Q$ is essential, $G$ is isomorphic to $\overline{F_8}$.
\end{enumerate}
The lemma now follows by a direct enumeration of all the possible cases for $k=3$ and $k=4$ taking into account the above facts.\end{proof}

Along this section, we will rely on some results about concave-round graphs. A graph is \emph{concave-round}~\cite{MR1760336} (sometimes also a \emph{$\Gamma$ circular-arc graph} or a \emph{Tucker circular-arc graph}) if there is a circular enumeration of its vertices such that the closed neighborhood of each vertex is an interval in the enumeration. The class of concave-round graphs was first studied by Tucker~\cite{MR0309810,MR0379298}, who proved the following.

\begin{thm}[\cite{MR0309810}]\label{thm:concave->CA} Every concave-round graph is a circular-arc graph.\end{thm}

As noticed in~\cite{MR1760336}, concave-round graphs can be recognized in linear time by means of the linear-time recognition algorithm for the circular-ones property devised in~\cite{MR0433962}.

\begin{thm}[\cite{MR1760336,MR0433962}]\label{thm:BoothLueker} Concave-round graphs can be recognized in linear time.\end{thm}

Our analysis will also rely on the following result concerning the minimal forbidden induced subgraphs for the class of concave-round graphs.

\begin{thm}[\cite{MR4043758}]\label{thm:concave-round} The minimal forbidden induced subgraphs for the class of concave-round graphs are: net, tent$^\ast$, $\overline{H_3}$, $\overline{\BII{1}}$, $\overline{\BII{2}}$, $\overline{\BIII{1}}$, $\overline{\BIII{2}}$, $\overline{\BIII{3}}$, $C_k^\ast$ for each $k\geq 4$, $\overline{C_{2k}}$ for each $k\geq 3$, and $\overline{C_{2k+1}^\ast}$ for each $k\geq 1$. Moreover, given a graph $G$ that is not concave-round, one of these minimal forbidden induced subgraphs contained in $G$ as an induced subgraph can be found in linear time.\end{thm}

A graph is \emph{quasi-line}~\cite{benrebea} if it contains no induced $\overline{C_{2k+1}^*}$ for any $k\geq 1$. The class of quasi-line graphs is a subclass of the class of graphs containing no induced claw and no induced 5-wheel. The equivalence of assertions \eqref{it:it1} and \eqref{it:it2} in the theorem below was noticed in~\cite{MR4043758}. 
By combining Theorems~\ref{thm:HCAobst} and \ref{thm:main}, Lemma~\ref{lem:claw-free-obst}, and Theorems~\ref{thm:concave->CA} and~\ref{thm:concave-round}, we are now able to extend the equivalence to assertion \eqref{it:it3}, which is the characterization by minimal forbidden induced subgraphs for the class of quasi-line Helly circular-arc graphs.

\begin{cor}\label{cor:quasi-HCA-forb} For each graph $G$, the following assertions are equivalent:
\begin{enumerate}[(i)]
 \item\label{it:it1} $G$ is concave-round and a Helly circular-arc;

 \item\label{it:it2} $G$ is quasi-line and a Helly circular-arc;
 
 \item\label{it:it3} $G$ contains no induced claw, $\overline{C_5^*}$, $\overline{C_7^*}$, $\overline{3K_2}$, $\overline{P_7}$, $\overline{F_1}$, $\overline{F_2}$, $\overline{H_3}$, net, $\overline{2P_4}$, $\overline{F_8}$, $\overline{C_6}$, tent$^*$, or $C_k^*$ for any $k\geq 4$.
\end{enumerate}\end{cor}
\begin{proof} \eqref{it:it1}${}\Rightarrow{}$\eqref{it:it2} Because concave-round graphs are quasi-line.

\eqref{it:it2}${}\Rightarrow{}$\eqref{it:it3} By Theorem~\ref{thm:HCAobst} and Lemma~\ref{lem:claw-free-obst} because $\overline{3K_2}$, $\overline{P_7}$, $\overline{F_1}$, $\overline{F_2}$, $\overline{H_3}$, net, $\overline{2P_4}$, $\overline{F_8}$ are obstacles and $\overline{C_6}$, tent$^*$, and $C_k^*$ for each $k\geq 4$ are not circular-arc graphs.

\eqref{it:it3}${}\Rightarrow{}$\eqref{it:it1} Let $G$ be a graph satisfying \eqref{it:it3}. Theorem~\ref{thm:concave-round} implies that $G$ is a concave-round graph because $\overline{\BII{k}}$ contains an induced $\overline{P_7}$ for each $k\in\{1,2\}$, $\overline{\BIII{k}}$ contains an induced $\overline{3K_2}$ for each $k\in\{1,2,3\}$, and each of $\overline{C_{2k}}$ and $\overline{C_{2k+1}^*}$ contains an induced $\overline{P_7}$ for each $k\geq 4$. In particular, $G$ is quasi-line. Moreover, by Theorem~\ref{thm:concave->CA}, $G$ is also a circular-arc graph. Hence, if $G$ were not a Helly circular-arc graph, then, by virtue of Theorem~\ref{thm:main}, $G$ would contain as an induced subgraph one of the essential obstacles listed in Lemma~\ref{lem:claw-free-obst}, contradicting either~\eqref{it:it3} or the fact that $G$ is quasi-line. (Notice that each of $\overline{F_3}$, $\overline{F_4}$, $\overline{F_5}$, $\overline{F_6}$, and $\overline{F_7}$ contains an induced $5$-wheel.) Therefore, $G$ is also a Helly circular-arc graph.\qedhere
\end{proof}

A circular-arc model $\mathcal A$ is \emph{proper} if no arc in the set $\mathcal A$ is strictly contained in another arc in $\mathcal A$. A \emph{proper circular-arc graph}~\cite{MR0309810} is a circular-arc graph having a proper circular-arc model. The results of Tucker in \cite{MR0309810,MR0379298} imply the following.

\begin{thm}[\cite{MR0309810,MR0379298}]\label{thm:Tucker} A graph is a proper circular-arc graph if and only if it is a $\{\overline{H_2},\overline{H_4}\}$-free concave-round graph.\end{thm}

Corollary~\ref{cor:quasi-HCA-forb} together with Theorem~\ref{thm:Tucker} leads to the result below, which characterizes the intersection of the classes of proper circular-arc graphs and Helly circular-arc graphs by minimal forbidden induced subgraphs. Notice that graphs in this intersection do not necessarily have circular-arc models which are proper and Helly simultaneously~\cite{MR2428582}.

\begin{cor} For each graph $G$, the following conditions are equivalent:
\begin{enumerate}[(i)]
 \item $G$ is a proper circular-arc graph and a Helly circular-arc graph;
 \item $G$ is quasi-line, $\{\overline{H_2},\overline{H_4}\}$-free, and a Helly circular-arc graph;

 \item $G$ contains no induced claw, $\overline{C_5^*}$, $\overline{C_7^*}$, $\overline{3K_2}$, $\overline{P_7}$, $\overline{F_1}$, $\overline{F_2}$, $\overline{H_2}$, $\overline{H_3}$, $\overline{H_4}$, net, $\overline{2P_4}$, $\overline{F_8}$, $\overline{C_6}$, tent$^*$, or $C_k^*$ for any $k\geq 4$.
\end{enumerate}\end{cor}

Our aim now is to extend Corollary~\ref{cor:quasi-HCA-forb} to the characterization of Helly circular-arc graphs by minimal forbidden induced subgraphs restricted to the class of graphs containing no induced claw and no induced 5-wheel.

We say that a graph $G$ is a \emph{multiple}~\cite{MR1271882} of a graph $H$ if $G$ arises from $H$ by replacing each vertex of $v$ of $H$ by a nonempty clique $Q_v$ and making two different cliques $Q_v$ and $Q_w$ complete (resp.\ anticomplete) in $G$ if and only if the corresponding vertices $v$ and $w$ are adjacent (resp.\ nonadjacent) in $H$. In particular, every graph is a multiple of itself. 
Let $Z$ be the graph depicted in Figure~\ref{fig:more-small}. Notice that $\overline Z$ is a Helly circular-arc graph and that any multiple of a Helly circular-arc graph is also a Helly circular-arc graph. The following lemma shows that the multiples of $\overline{C_7^*}$ and $\overline Z$ are the only Helly circular-arc graphs that contain no induced claw and no induced $5$-wheel but are not quasi-line.

\begin{lem}\label{lem:coC7*} If $G$ is a graph having an induced subgraph $J$ isomorphic to $\overline{C_7^*}$, then $G$ satisfies exactly one of the following assertions:
\begin{enumerate}[(i)]
 \item\label{it:VR1} it contains an induced claw, $5$-wheel, $C_4^*$, $\overline{3K_2}$, or $\overline{P_7}$;
 
 \item\label{it:VR2} it is a multiple of $\overline{C_7^*}$ or $\overline Z$.
\end{enumerate}
Moreover, given such a graph $G$, it can be decided in linear time whether or not $G$ satisfies assertion~\eqref{it:VR1} and, if it does, given $G$ and $J$, an induced subgraph of $G$ isomorphic to one of graphs listed in assertion~\eqref{it:VR1} can also be found in linear time.\end{lem}
\begin{proof} We say that an ordered partition $(V_1,\ldots,V_7,U,W)$ of $V(G)$ \emph{proves \eqref{it:VR2} for $G$} if, for each $i\in\{1,\ldots,7\}$, all the following assertions hold where, all along this proof, subindices on $V_i$ are modulo $7$:
\begin{enumerate}[$(1)$]
 \item $V_1,\ldots,V_7,U$ are nonempty;

 \item $V_1,\ldots,V_7,U,W$ are cliques;
 
 \item $V_i$ is complete to $V_{i-2},V_{i-1},V_{i+1},V_{i+2}$ and anticomplete to $V_{i-3}$ and $V_{i+3}$;

 \item $U$ is complete to $V_1,\ldots,V_7$;

 \item\label{it:RW} There is some $j\in\{1,\ldots,7\}$ such that $W$ is complete to $V_{j-2}$, $V_{j-3}$, $V_{j+3}$, and $V_{j+2}$, and anticomplete to $V_{j-1},V_j,V_{j+1},U$.
\end{enumerate}

The proof is by induction on the number $n$ of vertices of $G$. If $n=8$, $G$ is isomorphic to $\overline{C_7^*}$ and $G$ satisfies \eqref{it:VR2} trivially. Suppose $n\geq 9$ and that the lemma holds for graphs on $n-1$ vertices. Let $H$ be any induced subgraph of $G$ on $n-1$ vertices such that $V(J)\subseteq V(H)$. By the induction hypothesis, there is an ordered partition $\mathcal P_H=(V_1,\ldots,V_7,U,W)$ of $V(H)$ that proves \eqref{it:VR2} for $H$. For each $i\in\{1,\ldots,7\}$, let $v_i$ be an arbitrarily chosen vertex in the set $V_i$. Let $u$ be an arbitrarily chosen vertex in the set $U$. Let $x$ such that $V(G)-V(H)=\{x\}$.

We claim that either $G$ satisfies \eqref{it:VR1} or an ordered partition $\mathcal P_G$ that proves \eqref{it:VR2} for $G$ arises from $\mathcal P_H$ by adding vertex $x$ to one of the sets of $\mathcal P_H$. We analyze the possible neighbors and nonneighbors of $x$ in $H$ up to symmetry. In each case, we assume that none of the preceding cases hold.

\begin{enumerate}[{Case }1:]
  \item\label{casex:1} \emph{$x$ has nonneighbors $b_{i-2}\in V_{i-2}$ and $b_{i+2}\in V_{i+2}$ for some $i\in\{1,\ldots,7\}$.} On the one hand, if $x$ has some neighbor $b\in V_i\cup V_{i+3}$, then the set $\{b_{i-2},b_{i+2},x,b\}$ induces claw in $G$. On the other hand, if $x$ is anticomplete to $V_i\cup V_{i+3}$, then $\{b_{i-2},v_i,b_{i+2},v_{i+3},x\}$ induces $C_4^*$ in $G$.
  
  \item\label{casex:2} \emph{$x$ has nonneighbors $b_{i-1}\in V_{i-1}$ and $b_{i+1}\in V_{i+1}$ for some $i\in\{1,\ldots,7\}$.} As Case~\ref{casex:1} does not hold, $x$ is complete to $V_{i-2}$, $V_{i-3}$, $V_{i+3}$, and $V_{i+2}$. If $x$ has some neighbor $b\in V_i\cup U$, then $\{v_{i-2},b_{i-1},b_{i+1},v_{i+2},x,b\}$ induces $5$-wheel in $G$. Hence, we assume, without loss of generality, that $x$ is anticomplete to $V_i$ and $U$. If $x$ is adjacent to some $b_{i+1}'\in V_{i+1}$, then $\{u,x,v_i,v_{i-3},b_{i+1}',v_{i-2},v_{i+2}\}$ induces $\overline{P_7}$ in $G$. Hence, we assume, without loss of generality, that $x$ is anticomplete to $V_{i+1}$ and, symmetrically, also to $V_{i-1}$. Let $w\in W$ (if any) and let $j$ satisfying~\eqref{it:RW}. If $x$ is nonadjacent to $w$, then either $\{x,w,u,v_{j-2}\}$ or $\{x,w,u,v_{j+2}\}$ induces claw in $G$ because $x$ is not simultaneously nonadjacent to $v_{j+2}$ and $v_{j-2}$. Hence, we assume, without loss of generality, that $x$ is adjacent to $w$. If $V_i=V_{j+1}$ or $V_i=V_{j+2}$, then $\{x,w,v_{j+2},u,v_{j-1},v_{j-3}\}$ induces $5$-wheel in $G$. Symmetrically, if $V_i=V_{j-1}$ or $V_i=V_{j-2}$, then $\{x,w,v_{j-2},u,v_{j+1},v_{j+3}\}$ induces $5$-wheel in $G$. If $V_i=V_{j+3}$ or $V_i=V_{j-3}$, then
  $\{x,w,v_{j-3},u,v_j,v_{j-2}\}$ or $\{x,w,v_{j+3},u,v_j,v_{j+2}\}$ induces $5$-wheel in $G$, respectively. Hence, without loss of generality, if there is some $w\in W$, then $x$ is adjacent to $w$ and $j=i$. Therefore, the partition $\mathcal P_G$ that arises from $\mathcal P_H$ by adding $x$ to $W$ proves \eqref{it:VR2} for $G$.

    \item\label{casex:3} \emph{$x$ has nonneighbors $b_{i-3}\in V_{i-3}$ and $b_{i+3}\in V_{i+3}$}. As neither Case~\ref{casex:1} nor Case~\ref{casex:2} holds, $x$ is complete to $V_{i-2}$, $V_{i-1}$, $V_i$, $V_{i+1}$, and $V_{i+2}$. If $x$ has some neighbor $b_{i+3}'\in V_{i+3}$, then $\{x,b_{i-3},v_{i+1},v_{i-2},v_{i+2},v_{i-1},b_{i+3}'\}$ induces $\overline{P_7}$ in $G$. Hence, we assume, without loss of generality, that $x$ is anticomplete to $V_{i+3}$ and, by symmetry, also to $V_{i-3}$. If $x$ has some nonneighbor $b\in U$, then $\{b,x,v_{i+3},v_{i-1},v_{i+2},v_{i-2},v_{i+1}\}$ induces $\overline{P_7}$ in $G$. Hence, we assume, without loss of generality, that $x$ is complete to $U$. We will now prove that if there is some  $w\in W$ and $j$ satisfies~\eqref{it:RW}, then one of the following assertions holds:
    \begin{enumerate}[(a)]
     \item\label{it:a} $x$ is adjacent to $w$ if and only if $V_i\in\{V_{j-2},V_{j-3},V_{j+3},V_{j+2}\}$;
      
     \item\label{it:b} $G$ contains an induced claw, $5$-wheel, or $\overline{P_7}$.
    \end{enumerate}
    If $V_i=V_{j+1}$, then either $x$ is nonadjacent to $w$ or $\{v_{j+2},v_{j-2},x,v_{j-3},v_j,w,u\}$ induces $\overline{P_7}$ in $G$. Symmetrically, if $V_i=V_{j-1}$, then $x$ is nonadjacent to $w$ or $G$ contains an induced $\overline{P_7}$. If $V_i=V_{j+2}$ or $V_i=V_{j+3}$, then $x$ is adjacent to $w$ or $\{w,x,v_{j-1},v_{j-3}\}$ induces claw in $G$. Symmetrically, if $V_i=V_{j-2}$ or $V_i=V_{j-3}$, then $x$ is adjacent to $w$ or $G$ contains an induced claw. This completes the proof that either (a) or (b) holds for each $w\in W$ and each $j$ satisfying~\eqref{it:RW}. Therefore, the partition $\mathcal P_G$ that arises from $\mathcal P_H$ by adding $x$ to $V_i$ proves \eqref{it:VR2} for $G$.   

   \item\label{casex:4} \emph{$x$ has some nonneighbor $b_i\in V_i$.} As none of the preceding cases hold, $x$ is adjacent to $v_k$ for each $k\in\{1,\ldots,7\}-\{i\}$. Thus, $\{x,b_i,v_{i+3},v_{i-1},v_{i+2},v_{i-2},v_{i+1}\}$ induces $\overline{P_7}$ in $G$.

   \item\label{casex:5} \emph{$x$ is complete to $V_1\cup\cdots\cup V_7$.} If $x$ has a nonneighbor $b\in U$, then $\{x,b,v_{i-2},v_{i+1},v_i,v_{i+3}\}$ induces $\overline{3K_2}$ in $G$ for any $i\in\{1,\ldots,7\}$. Hence, we assume, without loss of generality, that $x$ is complete to $U$. Let $w\in W$ (if any) and let $j$ satisfying \eqref{it:RW}. If $x$ is adjacent to $w$, then $\{w,v_{j-2},v_{j-1},v_{j+1},v_{j+2},w,x\}$ induces $5$-wheel in $G$. Thus, we assume, without loss of generality, that $x$ is anticomplete to $W$. Therefore, the partition $\mathcal P_G$ that arises from $\mathcal P_H$ by adding $x$ to $U$ proves \eqref{it:VR2} for $G$.
\end{enumerate}
We have completed the proof of the claim and of the first assertion of the lemma.

Since the multiples of $\overline{C_7^*}$ and $\overline Z$ are Helly circular-arc graphs, $G$ satisfies exactly one of \eqref{it:VR1} and \eqref{it:VR2}. Hence, deciding whether $G$ satisfies \eqref{it:VR1} is equivalent to deciding whether $G$ does not satisfy \eqref{it:VR2}, which can be decided in linear time (e.g., by the algorithm for computing representative graphs in \cite{MR1354190}). If $G$ satisfies \eqref{it:VR1}, then a direct implementation of the inductive proof above gives a linear-time algorithm that, given $G$ and $J$, finds one of the induced subgraphs of $G$ listed in \eqref{it:VR1}.\end{proof}

We now give the main result of this section.

\begin{thm}\label{thm:claw-5wheel-algo} There is a linear-time algorithm that, given any graph $G$ that is not a Helly circular-arc graph, finds an induced subgraph of $G$ isomorphic to claw, $5$-wheel, or one of the following minimal forbidden induced subgraphs for the class of Helly circular-arc graphs: $\overline{3K_2}$, $\overline{P_7}$, $\overline{F_1}$, $\overline{F_2}$, $\overline{H_3}$, net, $\overline{2P_4}$, $\overline{F_8}$, $\overline{C_6}$, tent$^*$, or $C_k^*$ for any $k\geq 4$.\end{thm}
\begin{proof} Let $G$ be a graph that is not a Helly circular-arc graph. We first apply the algorithm of Theorem~\ref{thm:BoothLueker} to decide whether or not $G$ is concave-round.

Suppose first $G$ is concave-round. In particular, $G$ is quasi-line. Moreover, by virtue of Theorem~\ref{thm:concave->CA}, $G$ is also a circular-arc graph. We apply the algorithm of Corollary~\ref{cor:cAminHCA} to find an essential obstacle $H$ contained in $G$ as an induced subgraph. As $G$ is quasi-line, $H$ is quasi-line and belongs to the list of essential obstacles in Lemma~\ref{lem:claw-free-obst}. Hence, $H$ is isomorphic to one of the following graphs: $\overline{3K_2}$, $\overline{P_7}$, $\overline{F_1}$, $\overline{F_2}$, $\overline{H_3}$, net, $\overline{2P_4}$, or $\overline{F_8}$. We output $H$.

Suppose now that $G$ is not concave-round. Thus, we apply the algorithm of Theorem~\ref{thm:concave-round} to find a minimal forbidden induced subgraph $J$ for the class of concave-round graphs contained in $G$ as an induced subgraph. If $J$ is isomorphic to net, tent$^*$, $\overline{H_3}$, $\overline{C_6}$, or $C_k^*$ for some $k\geq 4$, we output $J$. If $J$ is isomorphic to $\overline{\BII{k}}$ for some $k\in\{2,3\}$ or to $\overline{C_{2k}}$ for some $k\geq 4$, we output an induced subgraph of $J$ isomorphic to $\overline{P_7}$. If $J$ is isomorphic to $\overline{\BIII{k}}$ for some $k\in\{1,2,3\}$, we output an induced subgraph of $J$ isomorphic to $\overline{3K_2}$. It only remains to consider the case where $J$ is isomorphic to $\overline{C_{2k+1}^*}$ for some $k\geq 1$. If $k\in\{1,2\}$, then $J$ is isomorphic to claw or $5$-wheel and we output $J$. If $k\geq 4$, then we output an induced subgraph of $J$ isomorphic to $\overline{P_7}$. Finally, if $k=3$, then we output an induced subgraph of $G$ isomorphic to claw, $5$-wheel, $C_4^*$, $\overline{3K_2}$, or $\overline{P_7}$ obtained through the algorithm of Lemma~\ref{lem:coC7*}. 

In all cases, we produce one of the induced subgraphs required by the statement of the theorem. The linear time bound for the whole procedure follows from the linear time bounds given in Corollary~\ref{cor:cAminHCA}, Theorems~\ref{thm:BoothLueker} and~\ref{thm:concave-round}, and Lemma~\ref{lem:coC7*}.\end{proof}

As a consequence, we obtain the minimal forbidden induced subgraph characterization for the class of Helly circular-arc graphs restricted to graphs containing no induced claw and no induced $5$-wheel.

\begin{cor} Let $G$ be a graph containing no induced claw and no induced $5$-wheel. Then, $G$ is a Helly circular-arc graph if and only if $G$ contains no induced $\overline{3K_2}$, $\overline{P_7}$, $\overline{F_1}$, $\overline{F_2}$, $\overline{H_3}$, net, $\overline{2P_4}$, $\overline{F_8}$, $\overline{C_6}$, tent$^*$, or $C_k^*$ for any $k\geq 4$.\end{cor}

\section*{Acknowledgements}

This work was partially supported by ANPCyT PICT 2012-1324 and PICT 2017-1315, CONICET PIO 14420140100027CO, and Universidad Nacional del Sur Grant PGI 24/L115.

\end{document}